\documentclass[12pt]{article}

\usepackage{amsmath,amsthm}
\usepackage{mathtools}
\usepackage{xcolor}
\usepackage{amssymb}
\usepackage{amsfonts}
\usepackage{graphicx}
\usepackage{booktabs}
\usepackage{units}
\usepackage{xcolor}
\usepackage{comment}
\usepackage{caption}
\usepackage{subcaption} 
\usepackage{dblfloatfix}   
\usepackage{algorithm}
\usepackage{algorithmicx}
\usepackage{algpseudocode}
\algrenewcommand\algorithmicindent{0.5em}

\usepackage{hyperref}
\hypersetup{
    colorlinks=true,
    linkcolor=blue,     
    urlcolor=blue,
}
\usepackage{xparse}

\usepackage{tikz}
\usepackage{pgfplots}
\usepackage{appendix}

\usepackage{enumitem}

\floatname{algorithm}{Pseudo code}
\DeclareMathOperator{\ran}{ran}

\newtheorem{definition}{Definition}

\newcounter{thmcounter}[section]  
\renewcommand{\thethmcounter}{\thesection.\arabic{thmcounter}}

\newcommand{\defthmwithqed}[2]{%
  \NewDocumentEnvironment{#1}{ o }{%
    \refstepcounter{thmcounter}
    \begin{trivlist}%
      \item[\hskip \labelsep \bfseries #2~\thethmcounter%
        \IfValueT{##1}{\ (##1)}.]%
  }{%
    \hfill $\square$%
    \end{trivlist}%
  }%
}

\defthmwithqed{theorem}{Theorem}
\defthmwithqed{lemma}{Lemma}
\defthmwithqed{assumption}{Assumption}
\defthmwithqed{modelass}{Modeling Assumption}
\defthmwithqed{proposition}{Proposition}
\defthmwithqed{prob}{Problem}
\defthmwithqed{corollary}{Corollary}
\defthmwithqed{remark}{Remark}
\defthmwithqed{example}{Example}

\usepackage{tikz}
\usetikzlibrary{arrows.meta,positioning,calc,decorations.pathmorphing}

{\left(\begin{smallmatrix}}
{\end{smallmatrix}\right)}

\newenvironment{smallbmatrix}
{\left[\begin{smallmatrix}}
{\end{smallmatrix}\right]}%

\title{Delay-dependent passivity and stability of linear port-Hamiltonian systems}
\author{Ikram El Haskouki\footnote{ENSAM, Hassan II University of Casabalanca, 150 Nile Boulevard, Casablanca 20670, Morocco (e-mail: \texttt{ikram.elhaskouki[at]gmail.com})} and Hannes Gernandt\footnote{University of Wuppertal, Gaußstraße 20, 42119 Wuppertal and Fraunhofer IEG, Fraunhofer Research Institution for Energy Infrastructures and Geotechnologies IEG, 03046 Cottbus, Germany  (e-mail: \texttt{gernandt[at]uni-wuppertal.de})} 
}

\date{\today}

\begin{document}

\maketitle

\begin{abstract}
We study delay-dependent and independent passivity properties of linear port-Hamiltonian systems with delay terms. We derive passivity conditions using a particular storage function and Wirtinger-based inequalities. Furthermore, we recall several 
conditions for delay-dependent and independent exponential stability. Finally, we present an application of our passivity and stability results to delay equations arising in the modeling of district heating networks and delayed-damping of oscillators.
\end{abstract}


\maketitle

\section{Introduction}
\label{sec:introduction}
Port-Hamiltonian systems provide an energy-based framework for the modeling, analysis, and control of physical systems. Their structure explicitly separates energy-conserving, dissipative, and port variables, which makes fundamental system-theoretic properties such as passivity, stability, and compositionality directly accessible. In particular, for finite-dimensional linear port-Hamiltonian systems of the form
\begin{align}
    \dot x(t)&=(J-R)Hx(t)+Bu(t),\\
    y(t)&=B^\top Hx(t),\nonumber
\end{align}
with $J=-J^\top$, $R\geq 0$, and $H>0$, the Hamiltonian $\frac12 x^\top Hx$ naturally defines a storage function. This immediately yields a passivity inequality with respect to the port variables $u$ and $y$. These structural properties are one of the main reasons why port-Hamiltonian systems are particularly useful for the modeling and control of interconnected physical systems such as electrical power systems \cite{FIAZ2013477,Schiffer2016,gernandt2024port} or district heating systems \cite{HauschildMarheinekeMehrmannetal.2019,Strehle}.

In many applications, however, delays occur naturally. They may arise from transport phenomena, communication networks, actuator and sensor dynamics, or from distributed-parameter effects that are approximated by finite-dimensional delay models. Such effects are relevant, for instance, in thermal and hydraulic networks, where transport delays play an essential role \cite{Li2025,HauschildMarheinekeMehrmannetal.2019,RoseTurnpike}. The inclusion of delayed state terms may substantially change the qualitative behavior of the system. In particular, properties that are immediate in the delay-free port-Hamiltonian case, such as passivity or stability, are in general no longer automatically preserved. Additional conditions are therefore required to determine whether a port-Hamiltonian system with delay terms remains passive or stable.

In this paper, we consider linear port-Hamiltonian systems with multiple constant delays of the form
\begin{align}
\label{eq:pH_delay}
    \dot x(t)
    &=(J-R)Hx(t)+\sum_{i=1}^q Z_iHx(t-\tau_i)+Bu(t), \quad t\geq 0,\\
    y(t)
    &=B^\top Hx(t),\qquad
    x(t)=\phi(t), \qquad t\in[-\tau_q,0],
    \nonumber
\end{align}
where $J=-J^\top\in\mathbb{R}^{n\times n}$, $R\geq 0$, $H>0$, $B\in\mathbb{R}^{n\times m}$, and $Z_i\in\mathbb{R}^{n\times n}$ for $i=1,\ldots,q$. The delays satisfy
\[
0<\tau_1\leq \cdots \leq \tau_q,
\]
and $\phi\in C([-\tau_q,0],\mathbb{R}^n)$ denotes the initial history. The matrices $Z_i$ describe the delayed interconnection terms and are not assumed to preserve the port-Hamiltonian dissipation structure. Hence, even if the undelayed part is passive, the full delay system need not be passive.

The inclusion of the delayed state terms in \eqref{eq:pH_delay} is not a
minor perturbation from the structural point of view. While the delay-free
part $(J-R)H$ has the standard port-Hamiltonian decomposition into an
energy-conserving and a dissipative part, the matrices $Z_i$ are general
delayed interconnection terms and need not preserve this decomposition.
Consequently, passivity, stability, and interconnection properties that are
immediate for delay-free port-Hamiltonian systems may be destroyed when
additional time-delay terms are added.

Port-Hamiltonian systems with delays have been considered in several
settings. Some works study delayed interconnections or delayed
port-Hamiltonian models motivated by particular applications, and often focus
primarily on stability properties; see, for instance,
\cite{KaoPasumarthy2012,SchifferFOR16,AouesLES2015}. Other recent work
develops abstract classes of port-Hamiltonian time-delay systems and
delay-independent passivity results \cite{BreitHU24}. However,
delay-dependent passivity criteria for finite-dimensional linear
port-Hamiltonian systems with general delayed state terms are less developed.

On the other hand, the theory of linear time-delay systems provides a rich
set of tools for stability and passivity analysis, including
Lyapunov--Krasovskii functionals, frequency-domain criteria, and integral
inequalities; see, e.g.,
\cite{HaleVerduynLunel1993,GuKharitonovChen2003,Fridman2014,
NiculescuLozano2001,FridmanShaked2002,Seuret2013, seuret2015hierarchy}. These methods are usually formulated for general
delay systems and therefore do not directly exploit the port-Hamiltonian structure. Thus, there remains a gap between structure-preserving
port-Hamiltonian analysis and delay-dependent passivity techniques for
general delayed systems.

The aim of this paper is to bridge this gap by deriving delay-dependent
passivity and stability conditions for linear port-Hamiltonian systems with
multiple constant delay terms. The proposed conditions combine
Lyapunov--Krasovskii storage functionals \cite{Fridman2014} and Wirtinger-based estimates \cite{Seuret2013} with
the port-Hamiltonian structure of the undelayed dynamics.

The main contributions of this paper are as follows.
\begin{itemize}
    \item We derive delay-dependent passivity conditions for linear port-Hamiltonian systems with multiple constant delays in form of a linear matrix inequality (LMI). 
    
    \item We show that the resulting passivity property is preserved under negative feedback interconnections. 
    
    \item We discuss delay-dependent and delay-independent stability of port-Hamiltonian delay systems. 
    
    \item We apply the theoretical results to delay equations arising in the modeling of district heating systems and to an oscillator example with delay-dependent damping and stiffness.
\end{itemize}

The remainder of the paper is organized as follows. In Section~\ref{sec:passivity}, we study passivity of port-Hamiltonian systems with delay terms and derive delay-dependent as well as delay-independent sufficient conditions. We also show that passivity is preserved under negative feedback interconnections. Section~\ref{sec:stability} is devoted to stability of linear port-Hamiltonian delay systems and recalls several delay-dependent and delay-independent stability criteria. Finally, the theoretical results are illustrated by examples motivated by district heating networks and delayed mechanical systems.

\subsection*{Notation}
Throughout this paper, $\|\cdot\|$ denotes the Euclidean norm in $\mathbb{R}^n$. For $x \in  C([-\tau,0],\mathbb{R}^{n}),$ and 
$t \geqslant 0,$ one defines the function  
$x_t \in  C([-\tau,0],\mathbb{R}^{n})$ by 
$x_t(\theta) = x(t+\theta),$ for all $\theta 
\in [-\tau, 0].$ The space $C([-
\tau,0],\mathbb{R}^{n})$ of continuous functions defined from $[-
\tau,0]$ to $\mathbb{R}^{n}$  is equipped with the norm
$ \Vert \psi \Vert_{\infty} = \sup\limits_{-\tau \leqslant 
s \leqslant 0} \Vert \psi(s)\Vert.$ Furthermore, the space $C^1([-
\tau,0],\mathbb{R}^n)$ denotes the set of all functions from $[-\tau,0]$ 
into $\mathbb{R}^n$ that are continuously differentiable.
The notation $P>0$ ($P \ge 0$), for $P \in \mathbb{R}^{n \times 
n}$ means that $P$ is symmetric and positive definite (positive 
semi-definite) and $\mathrm{sym}(A) := \frac{1}{2}(A + A^\top)$ denotes the symmetric part of a matrix $A \in \mathbb{R}^{n \times n }.$ We denote by $I_n$ the identity matrix in $\mathbb{R}^{n \times n}$ and $0$ represents the zero 
matrix of appropriate dimension.  
By $\mathrm{diag}(A_i)_{i=1}^{q}$ we denote a~block-diagonal matrix 
 with matrices $A_i$, for all $i=1,\ldots, q$ being its diagonal 
 blocks. In symmetric block matrices, the symbol $*$ stands for the symmetric term corresponding to the associated off-diagonal block.

\section{Passivity of port-Hamiltonian systems with delay terms}
\label{sec:passivity}

In this section, we study delay-dependent and independent passivity properties of linear port-Hamiltonian systems with delay terms.

\begin{definition}
The delay system \eqref{eq:pH_delay} is called \emph{passive for fixed delays $0<\tau_1\leq\ldots\leq\tau_q$} if there exists a storage function $\mathcal{S}:C^1([-\tau_q,0],\mathbb{R}^n)\rightarrow[0,\infty)$ with $\mathcal{S}(0)=0$ such that the following dissipation inequality holds for all $0\leq t_0\leq t_1<\infty$
\begin{align}
    \label{eq:dissip_ineq_cont_time}
\mathcal{S}(x_{t_1})-\mathcal{S}(x_{t_0})\leq \int_{t_0}^{t_1}y(s)^\top u(s)\mathrm{d}s,
\end{align}
and all initial histories $\phi\in C^1([-\tau_q,0],\mathbb{R}^n)$, 
where $x_t=x\vert_{[t-\tau_q,t]}\in C([-\tau_q,0],\mathbb{R}^n)$ is given by $x_t(s)=x(t+s)$ where $x$ is the solution to \eqref{eq:pH_delay} for some history $\phi$. 

If the dissipation inequality \eqref{eq:dissip_ineq_cont_time} holds for all $\tau\geq 0$, then we call the system \eqref{eq:pH_delay} \emph{passive independent of the delay.}
\end{definition}

In the following, we simplify the system representation without changing the passivity properties. Since $H$ is positive definite, we can perform a change of variables $\hat x:= H^{1/2} x$ and obtain 
\begin{align*}
H^{1/2}\dot x (t)&=\underbrace{H^{1/2}(J-R)H^{1/2}}_{=:\hat J-\hat R}\underbrace{H^{1/2} x(t)}_{=:\hat x(t)}\\&+ \sum_{i=1}^q\underbrace{H^{1/2}Z_iH^{1/2}}_{=:\hat Z_i}H^{1/2}x(t-\tau_i)+ H^{1/2}Bu(t).
\end{align*}
This leads to a pH system given in the following form 
\begin{align}
\label{eq:pH_delay_simple}    \dot x(t)&=\underbrace{(\hat J-\hat R)}_{=:\hat A}x(t)+\sum_{i=1}^q\hat Z_ix(t-\tau_i)+\hat Bu(t), \qquad  t \geq  0, \\
y(t)&=\hat B^\top x(t),\qquad \qquad  x(\tau)= \phi(\tau), \qquad \tau \in [-\tau_q , 0], \nonumber
\end{align}
which will be used to derive sufficient conditions on passivity. Since we perform a state-space transformation and do not change the input and output variables when going from \eqref{eq:pH_delay} to \eqref{eq:pH_delay_simple}, the passivity  of the two systems for a certain delay is equivalent to each other.

\subsection{Delay-dependent passivity}

Passivity of time-delay systems is typically studied 
via Lyapunov–Krasovskii functionals \cite{Fridman2014}, and in 
this paper we consider the following form 
\begin{align}  
 \nonumber 
\mathcal{S}(x_t)&=\xi_q(t)^\top  \mathcal{P} \xi_q(t)
 +\sum_{i=1}^q \int_{t-\tau_i}^t x(s)^\top \Theta_i x(s) \,\mathrm{d}s\\
& \quad + \sum_{i=1}^q \tau_i \int_{t-\tau_i}^t\int_s^t x(v)^\top \Omega_i x(v)\,\mathrm{d}v\,\mathrm{d}s,
\label{eq:storage_fixed_delay_PH} \\
&\quad + \sum_{i=1}^q \tau_i \int_{t-\tau_i}^t\int_s^t \dot{x}(v)^\top \Gamma_i\dot{x}(v)\,\mathrm{d}v\,\mathrm{d}s,\nonumber 
\end{align}
where $ \mathcal{P}, \Theta_i, \Gamma_i, \Omega_i \geq 0,$ with $ i=1,\ldots, q$ and 
\[
 \xi_q(t)=
\begin{bmatrix}
x(t)^\top  &\int_{-\tau_1}^{0} x_t(s)^\top  \,\mathrm{d}s &  \ldots &\int_{-\tau_q}^{0} x_t(s)^\top  \,\mathrm{d}s 
\end{bmatrix}^\top\in\mathbb{R}^{ n(q+1)} .
\]
\begin{remark}If $\Gamma_i=0$, for all $i=1,\ldots,q$ the storage function can be considered on the larger domain  $C([-\tau_q,0],\mathbb{R}^n)$.  \end{remark}
In the remainder, we aim to construct a passive output for the port-Hamiltonian delay system of the form 
\begin{align}
\label{eq:passive_long}
y_{\rm p}(t)=Cx(t)+\sum_{i=1}^q C_i x(t-\tau_i)+\sum_{i=1}^q \hat C_i\int_{-\tau_i}^0 x_t(s)\mathrm{d}s.
\end{align}
Using another extended state vector
\begin{align*}
\zeta_q(t)&=\left[
x(t)^\top ~
x(t-\tau_1)^\top ~ \cdots ~x(t-\tau_q)^\top  ~ ~\right.\\ &
\left.\tfrac{1}{\tau_1}\left(\int_{-\tau_1}^0 x_t(s)\mathrm{d}s\right)^\top  ~\cdots ~ \tfrac{1}{\tau_q}\left(\int_{-\tau_q}^0  x_t(s)\mathrm{d}s\right)^\top  ~u(t)^\top
\right]^\top
\end{align*}
the output equation can be cast in a more compact way as
\begin{align}
    \label{eq:def_passive_output}
y_{\rm p}(t)=\underbrace{\begin{bmatrix}
    C& C_1 & \ldots & C_q & \tau_1\hat{C}_1 & \ldots & \tau_q\hat{C}_q & 0
\end{bmatrix}}_{=:\mathcal{C} }\zeta_q(t).
\end{align}
In the following, we derive a Wirtinger-based inequality for semi-definite weights which essentially follows from the case of positive definite weights~\cite{Seuret2013,SeuretGF2013}, but this relaxation of the assumptions is necessary to study the passivity (cf.\ Remark~\ref{rem:passivity}).
\begin{lemma}\label{Lemma:Wirtinger}
For a positive semi-definite matrix $\Gamma\geq 0$ and any continuously differentiable function $\omega:[a,b]\rightarrow \mathbb{R}^n$ it holds
\begin{align}
(b-a)\int_a^b\dot \omega(\xi)^\top \Gamma\dot\omega(\xi)\mathrm{d}\xi&\geq (\omega(b)-\omega(a))^\top \Gamma (\omega(b)-\omega(a))+3\tilde{\omega}^\top \Gamma \tilde{\omega},  
\label{ineq:wirtinger}
\end{align}
with $\tilde{\omega}:=\omega(a)+\omega(b)-\frac{2}{b-a}\int_a^b\omega(\xi)\mathrm{d}\xi$.
\end{lemma}
\begin{proof}
For positive definite $\Gamma$ the estimate~\eqref{ineq:wirtinger} was proven in \cite[Lemma 3.6]{Fridman2014}. If $\Gamma$ is semi-definite, then we decompose the space $\mathbb{R}^n=\ker \Gamma \oplus \ran \Gamma$ and accordingly $\omega(t)=\omega_R(t)+\omega_K(t)$ and we can write $\omega_R(\xi)=\hat B\hat\omega(t)$ where $\hat B$ is a matrix whose columns are composed of basis vectors of $\ran \Gamma$. Moreover, $\dot \omega(t)=\dot\omega_R(t)+\dot\omega_K(t)$ and the closedness of the subspaces $\ker \Gamma$ and $\ran \Gamma$ implies that $\dot\omega_R(t)\in \ran \Gamma$ and $\dot\omega_K(t)\in \ker \Gamma$ holds for all $t\in[a,b]$. Summarizing, we find   
\begin{align*}
\int_a^b\dot \omega(\xi)\Gamma \dot\omega(\xi)\mathrm{d}\xi&=\int_a^b\dot \omega_R(\xi)\Gamma\dot\omega_R(\xi)\mathrm{d}\xi=\int_a^b\dot{\hat\omega}(\xi)\hat B^\top \Gamma \hat B\dot{\hat\omega}(t)\mathrm{d}\xi.
\end{align*}
Since $\hat B^\top \Gamma\hat B$ is positive definite and therefore we obtain the estimate from the positive definite result from \cite[Lemma 3.6]{Fridman2014} and using the decomposition again to obtain $\omega$ in the estimate \eqref{ineq:wirtinger}.
\end{proof}

\begin{theorem}
\label{thm:delay_passive}
The pH system with delay terms \eqref{eq:pH_delay_simple} is passive for the delays $0<\tau_1\leq\ldots \leq\tau_q$ with storage function $\mathcal{S}$ given by \eqref{eq:storage_fixed_delay_PH} if there exist positive semi-definite matrices $\mathcal{P}, \Theta_i,$ $\Omega_i$, $\Gamma_i$ such that
\begin{equation}\label{eq:pH_NLMI}
2{\rm sym}(G_2^\top \mathcal{P}G_1)+ M_1+ M_{2,\tau}+M_{3,\tau}+M_{4}+M_5 \leq 0,
\end{equation}
holds, where 
\begin{align*}
G_1&=\begin{bmatrix}
    I_n & 0 & 0&0\\
    0 &0& \text{diag} (\tau_i I_n)_{i=1}^q& 0\\
\end{bmatrix},\quad
G_2&=\begin{bmatrix}
\hat{A}&  \left[\begin{smallmatrix}\hat Z_1 &\ldots & \hat Z_q\end{smallmatrix}\right] & 0  & \hat{B}\\
\left[\begin{smallmatrix} I_n & \ldots &I_n \end{smallmatrix}\right]^\top & -I_{n\cdot q}& 0 & 0\\
\end{bmatrix},
\end{align*}
and the matrices $M_1$, $M_{2,\tau}$, $M_{3,\tau}$, $M_{4}$ and $M_5$ are defined below
\begin{align*}
 M_1= \begin{bmatrix}
      \sum_{i=1}^q \Theta_i& 0& 0 & 0 \\ 
     0 & -\mathrm{diag}(\Theta_i)_{i=1}^q & 0 & 0  \\
      0& 0 & 0 & 0\\
       0& 0 & 0 & 0
\end{bmatrix},
\end{align*}
and with $\Gamma_{\Sigma,\tau}:=\sum_{i=1}^q \tau_i^2  \Gamma_i$, $\Omega_{\Sigma,\tau}:=\sum_{i=1}^q \tau_i^2  \Omega_i$ and $\Gamma_{\Sigma,1}:=\sum_{i=1}^q  \Gamma_i$ we find 
\begin{align*}
M_{2,\tau}=\begin{bmatrix}
  \hat A^\top \Gamma_{\Sigma,\tau} \hat A & M_2^{12} & 0  &  \hat A^\top \Gamma_{\Sigma,\tau}\hat B \\
 *&  \begin{bmatrix}
\hat Z_i^\top   \Gamma_{\Sigma,\tau} \hat Z_j  
 \end{bmatrix}_{i,j=1}^q& 0 &(M_2^{42})^\top \\
 0&  0& -{\rm diag}(\tau_i^2 \Omega_{i})_{i=1}^q & 0\\
 *& 
*&0 & \hat B^\top   \Gamma_{\Sigma,\tau} \hat B 
\end{bmatrix}
\end{align*}
with
$M_2^{12}= \begin{bmatrix}
 \hat A^\top  \Gamma_{\Sigma,\tau} \hat Z_1 &  \ldots & \hat A^\top  \Gamma_{\Sigma,\tau}\hat Z_q 
\end{bmatrix}, \\
M_2^{42}=\begin{bmatrix}
\hat B^\top   \Gamma_{\Sigma,\tau} \hat Z_1 &  \ldots & \hat B^\top   \Gamma_{\Sigma,\tau} \hat Z_q 
 \end{bmatrix}$,
and
\begin{align*}
M_{3,\tau}&=\begin{bmatrix}
 -\Gamma_{\Sigma,1}+ \Omega_{\Sigma,\tau} & \begin{bmatrix}
\Gamma_1 & \ldots & \Gamma_q 
\end{bmatrix} & 0 & 0 \\
 *& -\mathrm{diag}(\Gamma_i)_{i=1}^q & 0 & 0 \\
 0 & 0& 0 & 0 \\
 0 & 0 &   0 & 0\\
\end{bmatrix},\\
 M_{4}&=-3\begin{bmatrix}
\Gamma_{\Sigma,1} & \!\!\!\!\begin{bmatrix}
 \Gamma_1 & \!\!\!\!\ldots \!\!\!\!& \Gamma_q 
\end{bmatrix}&\!\!\!\! -2\begin{bmatrix}
\Gamma_1& \!\!\!\!\ldots \!\!\!\!&\Gamma_q  
\end{bmatrix}& \!\!\!\!0\\
* & \mathrm{diag}(\Gamma_i)_{i=1}^q   & -2\mathrm{diag}( \Gamma_i)_{i=1}^q & \!\!\!\!0 \\
*& * & 4\mathrm{diag}(\Gamma_i)_{i=1}^q& \!\!\!\!0\\
0& 0& 0 & \!\!\!\!0 
\end{bmatrix},\\
M_5&=-{\rm sym}(\begin{bmatrix}
    0&\ldots&0&I_m
\end{bmatrix}^\top\mathcal{C}).  
\end{align*}
\end{theorem}
\begin{proof}
By considering the storage function \eqref{eq:storage_fixed_delay_PH}, we have
\begin{align}
   \frac{\rm d}{{\rm d}t} \mathcal{S}\big(x_{[t-\tau_q,\,t]}\big)&=
       2\xi_q(t)^\top \mathcal{P} \dot{\xi}_q(t)\label{eq:P_terms}\\
   &~~~~+ \sum_{i=1}^q \tfrac{\rm d}{{\rm d}t}\int_{t-\tau_i}^tx(s)^\top \Theta_i x(s)\mathrm{d}s \label{eq:Theta_term}\\
 &~~~~+  \sum_{i=1}^q \tau_i \tfrac{\rm d}{{\rm d}t} \int_{t-\tau_i}^t\int_s^t x(v)^\top \Omega_i x(v)\mathrm{d}v\mathrm{d}s \label{eq:Omega_term}\\
&~~~~+ \sum_{i=1}^q \tau_i \tfrac{\rm d}{{\rm d}t} \int_{t-\tau_i}^t\int_s^t \dot{x}(v)^\top \Gamma_i \dot{x}(v)\mathrm{d}v\mathrm{d}s. \label{eq:R_term}
    \end{align}
    We have
$\xi_q(t)=G_1\zeta_q(t),$
\begin{align*}\dot{\xi}_q(t)
&=\begin{bmatrix}
    \hat Ax(t)+\sum_{i=1}^q\hat Z_ix(t-\tau_i)+\hat Bu(t)\\ x(t)-x(t-\tau_1)\\ \vdots\\  x(t)-x(t-\tau_q)
\end{bmatrix}
=G_2
\zeta_q(t).
\end{align*}
Using the Leibniz rule for \eqref{eq:Theta_term}
and
\begin{align*}
&~~~~\sum_{i=1}^q  \tfrac{\rm d}{{\rm d}t}\int_{t-\tau_i}^t x(s)^\top \Theta_i x(s)\mathrm{d}s\\&=x(t)^\top \sum_{i=1}^q \Theta_i  x(t)- \sum_{i=1}^q x(t-\tau_i)^ \top \Theta_i x(t-\tau_i).
\end{align*}
For \eqref{eq:R_term}, we interchange the order of integration first
\begin{align*}
&~~~~~\sum_{i=1}^q \tau_i \int_{t-\tau_i}^t\int_s^t\dot{x}(v)^\top \Gamma_i \dot{x}(v) \mathrm{d}v\mathrm{d}s\\&=\sum_{i=1}^q \tau_i \int_{t-\tau_i}^t\int_{t-\tau_i}^v\dot{x}(v)^\top \Gamma_i \dot{x}(v)\mathrm{d}s\mathrm{d}v\\&=\sum_{i=1}^q \tau_i \int_{t-\tau_i}^t(v-(t-\tau_i))\dot{x}(v)^\top \Gamma_i \dot{x}(v)\mathrm{d}v
\end{align*}
and we apply Leibniz rule 
\begin{align}\nonumber
&~~~~\sum_{i=1}^q \tau_i   \tfrac{\rm d}{{\rm d}t} \int_{t-\tau_i}^t\int_s^t \dot{x}(v)^\top \Gamma_i \dot{x}(v) \mathrm{d}v\mathrm{d}s\nonumber \\&=\sum_{i=1}^q \tau_i   \tfrac{\rm d}{{\rm d}t} \int_{t-\tau_i}^t(v-(t-\tau_i))\dot{x}(v)^\top \Gamma_i \dot{x}(v)\mathrm{d}v \nonumber\\
&=\sum_{i=1}^q \tau_i^2 \dot{x}(t)^\top  \Gamma_i  \dot{x}(t)
-\sum_{i=1}^q \tau_i\int_{t-\tau_i}^t\dot{x}(s)^\top \Gamma_i  \dot{x}(s)\mathrm{d}s \label{eq:deriv_R_term}.
\end{align}
To estimate the integral in~\eqref{eq:deriv_R_term}, we use Wirtinger's inequality from Lemma~\ref{Lemma:Wirtinger}
\begin{align*}\nonumber
&~~~~\tau_i\int_{-\tau_i}^{0}  \dot{x}_t(u)^\top \Gamma_i \dot{x}_t(u)du \\&\geq  \hat \xi_i(t)^\top \begin{bmatrix}
    I_n & -I_{n} &0 & 0
\end{bmatrix}^\top \Gamma_i  \begin{bmatrix}
    I_n & -I_{n} &0 & 0
\end{bmatrix}\hat \xi_i(t)\\
&+3
\hat \xi_i(t)^\top\begin{bmatrix}
    I_n & I_{n} &-2I_{n} & 0
\end{bmatrix}^\top
\Gamma_i\begin{bmatrix}
    I_n & I_n &-2I_{n} & 0
\end{bmatrix}\hat\xi_i(t)
\end{align*}
and using $\hat  \xi_i(t)=\begin{bmatrix}
x(t)^\top & x(t-\tau_i))^\top &  \tfrac{1}{\tau_i}(\int_{t-\tau_i}^{t} x(s)\,ds
)^\top & u(t)^\top\end{bmatrix}^\top$ yields
{\begin{align*}\nonumber 
& ~~~~\tau_i \frac{\rm d}{{\rm d}t}\int_{t-\tau_i}^t\int_s^t \dot{x}(v))^\top \Gamma_i\dot{x}(v)\mathrm{d}v\mathrm{d}s \\& \leq \hat \xi_i(t)^\top \begin{bmatrix}
   \hat{A} & \sum_{i=1}^q \hat{Z_i}  & 0 & \hat{B}
\end{bmatrix}^\top  \tau_i^2\Gamma_i  \begin{bmatrix}
    \hat{A}&   \sum_{i=1}^q\hat{Z_i} & 0 & \hat{B}
\end{bmatrix}\hat \xi_i(t)\\
&-  \hat \xi_i(t)^\top  \begin{bmatrix}
    I_n& -I_{n} &0 & 0
\end{bmatrix}^\top \Gamma_i \begin{bmatrix}
    I_n & -I_{n} &0 & 0
\end{bmatrix}\hat \xi_i(t)\\
&- 3\sum_{i=1}^q \hat \xi_i(t)^\top \begin{bmatrix}
    I _n& I_{n} &-2I_{n} & 0
\end{bmatrix}^\top
\Gamma_i
\begin{bmatrix}
    I_n & I_{n} &-2I_{n} & 0
\end{bmatrix}\hat \xi_i(t).
\nonumber
\end{align*}}
Similarly, we can rewrite and estimate the term \eqref{eq:Omega_term} using Jensen's inequality
{\small\begin{align*}
&~~~~\sum_{i=1}^q \tau_i   \tfrac{\rm d}{{\rm d}t} \int_{t-\tau_i}^t\int_s^t x(v)^\top \Omega_i x(v) \mathrm{d}v\mathrm{d}s\nonumber \\&
=\sum_{i=1}^q\tau_i^2 x(t)^\top \Omega_i x(t)-\sum_{i=1}^q\tau_i\int_{t-\tau_i}^t x(s)^\top \Omega_i x(s)\mathrm{d}s\\
& \leq \sum_{i=1}^q\tau_i^2 x(t)^\top \Omega_i x(t)- \sum_{i=1}^q\left(\int_{t-\tau_i}^t x(s)\mathrm{d}s\right)^\top \Omega_i\int_{t-\tau_i}^t x(s)\mathrm{d}s.
\end{align*}}
Summarizing, this leads to 
\begin{align*}
& ~~~~ \frac{\rm d}{{\rm d}t} \mathcal{S}\big(x_{[t-\tau_q,\,t]}\big)\\&\leq \zeta_q(t)^\top G_1^\top \mathcal{P}G_2 \zeta_q(t)+\zeta_q(t) G_2^\top \mathcal{P}G_1\zeta_q(t)\\
&+\sum_{i=1}^q \hat\xi_i(t)^\top \begin{bmatrix}
  I_n &0&0&0  
\end{bmatrix}^\top \Theta_i  \begin{bmatrix}
  I_n &0&0&0  
\end{bmatrix} \hat \xi_i(t)\\&- \sum_{i=1}^q \hat \xi_i(t)^\top \begin{bmatrix}
   0 & I_{n} & 0 & 0 
\end{bmatrix}^\top \Theta_i  \begin{bmatrix}
   0 & I_{n} & 0 & 0 
\end{bmatrix} \hat \xi_i(t)\\
&+\sum_{i=1}^q \tau_i^2    \hat \xi_i(t)^\top \begin{bmatrix}
       I_n&0& 0 & 0
\end{bmatrix}^\top  \Omega_i  \begin{bmatrix}
    I_n&   0& 0 & 0
\end{bmatrix}\hat  \xi_i(t)\\
&-\sum_{i=1}^q \tau_i^2    \hat \xi_i(t)^\top \begin{bmatrix}
       0&0& I_{n} & 0
\end{bmatrix}^\top  \Omega_i  \begin{bmatrix}
    0&   0& I_{n} & 0
\end{bmatrix}\hat  \xi_i(t)\\
&-\sum_{i=1}^q \hat \xi_i(t)^\top \begin{bmatrix}
    I_n & -I_{n} &0 & 0
\end{bmatrix}^\top \Gamma_i\begin{bmatrix}
    I_n & -I_{n} &0 & 0
\end{bmatrix}\hat \xi_i(t)\\
&+\sum_{i=1}^q \tau_i^2    \hat \xi_i(t)^\top \begin{bmatrix}
        \hat A& \sum_{i=1}^q \hat Z_i& 0 & \hat B
\end{bmatrix}^\top  \Gamma_i  \begin{bmatrix}
    \hat A &   \sum_{i=1}^q \hat Z_i & 0 & \hat B
\end{bmatrix}\hat  \xi_i(t)\\
&-3\sum_{i=1}^q  \hat \xi_i(t)^{ \top} \begin{bmatrix}
I_n & I_{n} &-2I_{n} & 0
\end{bmatrix}
^\top
\Gamma_i
\begin{bmatrix}
    I_n & I_{n} &-2I_{n} & 0
\end{bmatrix}\hat \xi_i(t)\\
&+y(t)^\top u(t)-y(t)^\top u(t)\\
&=
 \zeta_q(t)^\top \text{$\left[2\rm sym(G_1^\top \mathcal{P}G_2)+ M_1  +  M_{2,\tau}+ M_{3,\tau}+M_{4}+M_5\right]$}\zeta_q(t)\\
&\leq y(t)^\top u(t).
\end{align*}
Using \eqref{eq:pH_NLMI}, we conclude that 
\begin{equation}\label{eq:pH_inequality}
\frac{d}{dt} \mathcal{S}\big(x_{[t-\tau_q,\,t]}\big) \leq y(t)^\top u(t).
\end{equation}
Finally, integrating \eqref{eq:pH_inequality} over the interval $[t_0,t_1]$ yields the dissipation inequality. 
\end{proof}

\begin{remark}
\label{rem:passivity}
\begin{itemize}
    \item[(a)] If we choose $\Gamma_i=0$, $\Omega_i=0$ and $\mathcal{P}={\rm diag}(Q,0,\ldots,0)$ then we recover the storage function that was used for delay-independent passivity analysis in \cite{BreitHU24}. If we use the colocated output $y(t)=\hat B^\top x(t)$, then looking at the entry $\hat B^\top \Gamma_{\Sigma,\tau} \hat B$ in the block matrix $M_{2,\tau}$ gives a restriction on the coefficient matrices $\Gamma_i$ of the form $\ran B\subseteq \bigcap_{i=1}^q\ker\Gamma_i$. 
\item[(b)]The domain for the delays for which passivity holds can be improved by using the Bessel–Legendre inequality \cite{seuret2015hierarchy} instead of the Wirtinger inequality. This leads to an LMI similar to~\eqref{eq:pH_NLMI} with larger block matrices. 
\end{itemize}
\end{remark}

\subsection{Delay-independent passivity}

In this section we study the delay-independent passivity
using the storage function
\begin{align}
\label{eq:storage_independent}
\mathcal{S}(x_t)=x(t)^\top  P x(t)
 +\sum_{i=1}^q \int_{t-\tau_i}^t x(s)^\top \Theta_i x(s) \,\mathrm{d}s
\end{align}
for some positive semi-definite matrices $P,\Theta_1,\ldots,\Theta_q\in\mathbb{R}^{n\times n}$.

As a special case of Theorem~\ref{thm:delay_passive}, we obtain a sufficient condition for delay-independent passivity of pH systems \eqref{eq:pH_delay_simple}.

\begin{corollary}
\label{cor:independent}
   The pH system with delay terms~\eqref{eq:pH_delay_simple} is passive independently of the delay with the storage function~\eqref{eq:storage_independent} if there exist positive semi-definite matrices $P,\Theta_1,\ldots,\Theta_q\in\mathbb{R}^{n\times n}$ such that 
\[\begin{bmatrix}
      \sum_{i=1}^q \Theta_i-\text{sym}(P\hat{R}) & \tfrac{1}{2}P\begin{bmatrix}
  \hat Z_1  &\ldots  &\hat Z_q
\end{bmatrix}  \\ 
   * & -\mathrm{diag}(\Theta_i)_{i=1}^q  \\
\end{bmatrix}\leq 0.\]
\end{corollary}

Note that in \cite{BreitHU24} a similar result for pH delay systems \eqref{eq:pH_delay} but with choosing $P=H^{-1}$ where $H$ is the Hamiltonian weight in \eqref{eq:pH_delay}. Allowing the choice of a positive semi-definite weight $P$ that is not related to the Hamiltonian can make a difference, since the delay-independent passivity condition requires the $R$ in the pH system representation to be positive definite. This is problematic when the Hamiltonian weight $H$ does not fulfill a Lyapunov equation. 

However, if $(J-R)H$ in \eqref{eq:pH_delay} is Hurwitz, then we can find a coordinate transformation $\hat H$ in such a way that $\hat R$ in \eqref{eq:pH_delay_simple} is positive definite. Indeed, consider a positive definite solution $\hat H$ to the Lyapunov inequality
\[
\hat H (J-R)H +((J-R)Q)^\top \hat H\leq -2\hat H.
\]
This inequality is equivalent to 
\[
\hat H^{\tfrac{1}{2}} (J-R)H\hat H^{-\tfrac{1}{2}} +\hat H^{-\tfrac{1}{2}}((J-R)H)^\top \hat H^{\tfrac{1}{2}}\leq -2 I_n.
\]
After coordinate transformation $\hat x=\hat H^{\tfrac{1}{2}}x$ we obtain a pH system \eqref{eq:pH_delay_simple}
\begin{align*}
\hat J&=\tfrac{1}{2}(\hat H^{1/2} (J-R)H\hat H^{-\tfrac{1}{2}}-(\hat H^{\tfrac{1}{2}} (J-R)H\hat H^{-\tfrac{1}{2}})^\top),\\ \hat R&=-\tfrac{1}{2}(\hat H^{\tfrac{1}{2}} (J-R)H\hat H^{-\tfrac{1}{2}}+(\hat H^{\tfrac{1}{2}} (J-R)H\hat H^{-\tfrac{1}{2}})^\top).
\end{align*}
Since this transformed system has Hamiltonian the Hamiltonian weight $H=I_n$ we recover the conditions from \cite{BreitHU24} by choosing  $P=I_n$ leading to the delay-independent passivity. On the other hand, this passivity can be verified directly based on Corollary~\ref{cor:independent} without a change of coordinates by a suitable choice of $P$ as the solution to a Lyapunov equation.

\subsection{Passivity implies positive real transfer function}
In this section, we study the transfer function of linear delay systems of the form
\begin{align}
\label{eq:linear_delay}
    \dot x(t)&=A_0 x(t)+\sum_{i=1}^qA_i x(t-\tau_i)+Bu(t), \quad t\geq 0,\\\nonumber 
    y(t) &=C x(t).\qquad
\end{align}
Then the transfer function is given by 
\begin{align}
\label{eq:transfer_function}    
G(s)=C\left(s I_n -A_0-\sum_{i=1}^q{\rm e}^{-\tau_i s}A_i\right)^{-1}B.
\end{align}

Is well known for LTI systems that passivity implies a~particular property of the transfer function being positive real \cite{CherifiGernandtHinsen2024}. We recall this property.
\begin{definition}
Let $G_\tau(s)$ be a function with values in $\mathbb{C}^{n\times n}$.  Then $G_\tau (s)$
is called \emph{positive real} if $G_{\tau}(s)$ has no poles in the right half-plane $\mathbb{C}_+:=\{ s \in \mathbb{C}: \mathrm{Re}(s)>0\}$ 
and $G_\tau(s)+G_\tau(s)^*\geq 0,$ for all $s\in \mathbb{C}_+. $
\end{definition}

In the following proposition, we show that the delay-dependent passivity implies positive realness of the transfer function. This result follows immediately by adapting the result for single delay systems  \cite[Proposition 4.5]{Fridman2014} to multiple delays.
\begin{proposition}
\label{prop:positive real}
If the linear delay system \eqref{eq:linear_delay} 
is passive for some delays $\tau_q\geq \ldots\geq \tau_1>0$ then the transfer function $G$ given by \eqref{eq:transfer_function} is positive real and 
for all $\omega>0$ for which 
\begin{align}
    \label{eq:characteristic}
\det \left(s I_n -A_0-\sum_{i=1}^q{\rm e}^{-\tau_i s}A_i\right)=0,
\end{align} with $s=i\omega$, 
we have that $G(-i\omega)^\top+G(i\omega)\geq 0$.
\end{proposition}

The Proposition~\ref{prop:positive real} is useful to conclude that a system is not passive. This is illustrated by the following example of a~pH systems with delay terms~\eqref{eq:pH_delay} which is not passive but for which we find a passive output of the form \eqref{eq:def_passive_output}. 
\begin{example}
Consider the scalar pH system \eqref{eq:pH_delay} with a single delay term, i.e.\ $q=1$ and $\tau_1=1$,
\begin{align}
\begin{split}
        \dot x(t)&=-\frac12 x(t)-x(t-1)+u(t),\\
    y(t)&=x(t).
\end{split}
    \label{eq:delay_example}
\end{align}
The transfer function from $u$ to $y$ is  $G(s)=(s+\frac12+e^{-s})^{-1}$.
We use Proposition~\ref{prop:positive real}  to conclude that the system is not passive. To this end we compute 
\[
    G(i\omega)
    =
    \frac{1}{i\omega+\frac12+e^{-i\omega}}
    =
    \frac{1}{\frac12+\cos\omega+i(\omega-\sin\omega)}.
\]
At $\omega=\pi$, this gives
\[
G(i\pi)+G(-i\pi)
    =
    \frac{\frac12+\cos\pi}
    {
    \left(\frac12+\cos\pi\right)^2+\pi^2
    }
    =
    \frac{-\frac12}{\frac14+\pi^2}
    <0.
\]
Therefore, $G$ is not positive real. Hence, the system \eqref{eq:delay_example} is not passive by Proposition~\ref{prop:positive real}.

We now construct a passive output of the form \eqref{eq:def_passive_output}. Define
\[
    y_{\rm p}(t) =\mathcal{C}\zeta_1(t)=\begin{bmatrix}
        1 &0&-1 &0
    \end{bmatrix} \zeta_1(t)
    =x(t)-\int_{t-1}^{t}x(s)\,\mathrm{d}s .
\]
This corresponds to 
$C=1$, $C_1=0$, $\hat C_1=-1$, and, with
\[
    \zeta_1(t)
    =
    \begin{bmatrix}
        x(t)& x(t-1)& \int_{t-1}^{t}x(s)\,\mathrm{d}s& u(t)
    \end{bmatrix}^\top .
\]
We apply Theorem~\ref{thm:delay_passive}, i.e.\ the delay-dependent
passivity result.
Choose
\[
    \mathcal P
    =
    \begin{bmatrix}
        \frac12&-\frac12\\[1mm]
        -\frac12&\frac12
    \end{bmatrix},
    \qquad
    \Theta_1=0,
    \qquad
    \Gamma_1=0,
    \qquad
    \Omega_1=\frac34 .
\]
Then $\mathcal P\ge0$, $\Theta_1\ge0$, $\Gamma_1\ge0$, and
$\Omega_1\ge0$. The corresponding storage functional is
\begin{align}
\label{eq:storage_example_lmi}
    \mathcal S(x_t)
    &=
    \begin{bmatrix}
        x(t)\\ \int_{t-1}^{t}x(s)\,\mathrm{d}s
    \end{bmatrix}^{\!\top}
    \mathcal P
    \begin{bmatrix}
        x(t)\\ \int_{t-1}^{t}x(s)\,\mathrm{d}s
    \end{bmatrix}
    +
    \frac34
    \int_{t-1}^{t}\int_s^t x(v)^2\,\mathrm{d}v\,\mathrm{d}s.
\end{align}

For the chosen data, the matrices $G_1$ and $G_2$ are
\[
    G_1
    =
    \begin{bmatrix}
        1&0&0&0\\
        0&0&1&0
    \end{bmatrix},
    \qquad
    G_2
    =
    \begin{bmatrix}
        -\frac12&-1&0&1\\
        1&-1&0&0
    \end{bmatrix}.
\]
Since $\Theta_1=\Gamma_1=0$, we obtain for the matrices  $M_{1}$, $M_{2,\tau}$, $M_{3,\tau}$, $M_{4}=0$, and $M_5$ in  
Theorem~\ref{thm:delay_passive} that $M_{4}=0$ and 
\begin{align*}
        M_{1}+M_{2,\tau}+M_{3,\tau}
    &=
    \begin{bmatrix}
        \frac34&0&0&0\\
        0&0&0&0\\
        0&0&-\frac34&0\\
        0&0&0&0
    \end{bmatrix},\\ 2\operatorname{sym}(G_2^\top\mathcal P G_1)&=\begin{bmatrix}
      -\frac32&0&\frac34&\frac12\\
        0&0&0&0\\
        \frac34&0&0&-\frac12\\
        \frac12&0&-\frac12&0
    \end{bmatrix}
\end{align*}
and
\[
    M_5
    =
    -\operatorname{sym}
    \left(
        \begin{bmatrix}
            0&\!\!0&\!\!0&\!\!1
        \end{bmatrix}^{\!\top}
        \mathcal C
    \right)=-\frac{1}{2}\begin{bmatrix}
      0&0&0&1\\
        0&0&0&0\\
        0&0&0&-1\\
        1&0&-1&0
    \end{bmatrix}.
\]
This leads to 
\[
    2\operatorname{sym}(G_2^\top\mathcal P G_1)
    +
    M_{1}+M_{2,\tau}+M_{3,\tau}
    +
    M_5\!=\!\begin{bmatrix}
        -\frac34&\!\!0&\!\!\frac34&\!\!0\\
        0&\!\!0&\!\!0&\!\!0\\
        \frac34&\!\!0&\!\!-\frac34&\!\!0\\
        0&\!\!0&\!\!0&\!\!0
    \end{bmatrix}
    \!\le\!0.
\]
Therefore, all assumptions of Theorem~\ref{thm:delay_passive} are
satisfied. Consequently, the system \eqref{eq:delay_example} is passive
with respect to the modified output $y_{\rm p}$ 
and the storage functional \eqref{eq:storage_example_lmi}. 
This shows that the delayed pH system is not passive for its natural
colocated output $y=x$, but it is passive for the constructed output
$y_{\rm p}$ of the form \eqref{eq:def_passive_output}.
\end{example}

\subsection{Passivity-preserving interconnection}
In this section, we show that the delay-dependent passivity properties are preserved under negative feedback interconnections. For a related result for delay-independent passivity we refer to \cite{BreitHU24}. To this end we consider two passive pH systems with delay terms of the form
{\small\begin{align}\label{eq:Interconnection}
\begin{split}
 \dot{x}_j(t)&= (J_j - R_j)H_jx_j(t)+\sum_{i=1}^q Z_{i,j}H_jx_j(t-\tau_{i,j}) + B_j u_j(t),\\
y_j (t)&= B_j^{\top}H_j x_j(t),\quad j=1,2,
\end{split}
\end{align}}
with storage functions $\mathcal{S}_1$ and $\mathcal{S}_2$, respectively.

The following result shows that passivity is preserved under negative feedback interconnections. For a related result with a slightly different passivity notion, see~ \cite{kawano2023passivity}.

\begin{proposition}
If the systems~\eqref{eq:Interconnection} are passive for a fixed delay $\tau$ with storage functions $\mathcal{S}_1$ and $\mathcal{S}_2$ then the negative feedback interconnection
\[
u_1=v-y_2, \quad u_2=y_1,
\]
with external input $v$ is 
is passive with storage function $\mathcal{S}=\mathcal{S}_1+\mathcal{S}_2$.  
\end{proposition}
\begin{proof}
Since the passivity inequalities hold individually for arbitrary inputs which gives
\begin{align*}
&~~~~\mathcal{S}(x_{t_1})-\mathcal{S}(x_{t_0})\\&=\mathcal{S}_1(x_{t_1}^1)-\mathcal{S}_1(x_{t_0}^1)+\mathcal{S}_2(x_{t_1}^2)-\mathcal{S}_2(x_{t_0}^2) \\ 
&\leq \int_{t_0}^{t_1}y_1(s)^\top u_1(s)\mathrm{d}s+\int_{t_0}^{t_1}y_2(s)^\top u_2(s)\mathrm{d}s\\
&=\int_{t_0}^{t_1}y_1(s)^\top (v-y_2)(s)\mathrm{d}s+\int_{t_0}^{t_1}y_2(s)^\top y_1(s)\mathrm{d}s\\
&=\int_{t_0}^{t_1}y_1(s)^\top v(s)\mathrm{d}s
\end{align*}
which is the desired passivity for the interconnected system which has input $v$ and output $y_1$.
\end{proof}

\section{Stability of delay port-Hamiltonian systems}
\label{sec:stability}
In this section, we recall some characterizations and sufficient conditions on stability of the zero-input system \eqref{eq:linear_delay}, i.e.\ $u=0$.
This system is called \emph{(globally) exponentially stable} if there exist constants $b>0$ and $c\geq 1$ such that for all initial histories $\phi\in C([-\tau,0],\mathbb{R}^n)$ the following estimate holds
\begin{align}
\label{exp_stability}    
\|x(t)\|\leq ce^{-bt}\|\phi\|_\infty \quad \text{for all $t\geq 0$}.
\end{align}
Moreover, the system \eqref{eq:linear_delay} is called \emph{locally asymptotically stable} if there exists $\delta>0$ such that for all initial histories $\phi\in C([-\tau,0],\mathbb{R}^n)$ with $\|\phi\|_\infty<\delta$ we have $\|x(t)\|\rightarrow 0$ as $t\rightarrow\infty$.

The following equivalences regarding the stability of linear delay systems are well known, see \cite[Lemma 5.3 in Ch. 6]{HaleVerduynLunel1993} and \cite[Corollary 2.1]{Fridman2014}. 

\begin{proposition}
\label{prop:stability_equivalence}
For the linear delay system \eqref{eq:linear_delay} the following statements are equivalent:
\begin{itemize}
    \item[\rm (i)] The system \eqref{eq:linear_delay} is locally asymptotically stable.
    \item[\rm (ii)] The system \eqref{eq:linear_delay} is globally exponentially stable.
    \item[\rm (iii)] All roots $s\in\mathbb{C}$ of the characteristic equation~\eqref{eq:characteristic} 
have negative real parts.
\end{itemize}
\end{proposition}

Note that the exponential stability of  linear port-Hamiltonian time-delay systems \eqref{eq:pH_delay} is equivalent to that of the transformed system \eqref{eq:pH_delay_simple}. Therefore, we will consider in the following only the system \eqref{eq:pH_delay_simple}.

First, we recall the 
\emph{frequency sweeping test} which was obtained for linear systems with a single delay in \cite{gu2003stability} to establish stability independently of the delay. Here, we reformulate the result for a port-Hamiltonian system and including multiple delays. The result follows from a combination of \cite[Corollary 2.1]{Fridman2014} and \cite[Corollary 1]{NiculescuChen1999} and we denote by $\rho(M)$ the spectral radius of a matrix $M\in\mathbb{C}^{n\times n}$ which is the maximum modulus of its eigenvalues.
\begin{proposition}
\label{prop:sweeping}
The system \eqref{eq:pH_delay_simple} is exponentially stable independently of the delay if and only if the following conditions hold
\begin{itemize} 
\item[\rm (i)] $\hat J-\hat R$ is Hurwitz;
    \item[\rm (ii)] $\hat J-\hat R+\sum_{i=1}^q \hat Z_i$ is Hurwitz;
    \item[\rm (iii)] $\rho\left(\begin{smallbmatrix}
        I_n&\ldots& I_n
    \end{smallbmatrix}^\top(j\omega I-(\hat J-\hat R))^{-1} \begin{smallbmatrix}
        \hat Z_1 &\ldots& \hat Z_q 
    \end{smallbmatrix} \right)<1$, for all $\omega >0$. 
    \end{itemize}   
    \end{proposition}
\begin{remark}
Let $\|\cdot\|$ be any matrix norm that is compatible with a vector norm then $\rho(M)\leq\|M\|$ holds. In particular, we can use the row-sum norm of a matrix $M\in\mathbb{R}^{n\times m}$ given by  $\|M\|_\infty:=\max_{i=1,\ldots,n}\sum_{j=1^m}|m_{ij}|$ to obtain a sufficient condition for Proposition~\ref{prop:sweeping}~(iii) which is 
\begin{align*}
    &~~~~\|\begin{bmatrix}
        I_n&\ldots& I_n
    \end{bmatrix}^\top(j\omega I-(\hat J-\hat R))^{-1} \begin{bmatrix}
        \hat Z_1 &\ldots& \hat Z_q 
    \end{bmatrix}\|_\infty
    \\
    &=\|(j\omega I-(\hat J-\hat R))^{-1} \begin{bmatrix}
        \hat Z_1 &\ldots& \hat Z_q 
    \end{bmatrix}\|_\infty<1.
\end{align*}
\end{remark}

For an overview of delay-dependent and independent stability conditions, we refer to \cite{Fridman2014,Scholl2024}. Below we present a condition on delay-dependent exponential stability which is based on the LMI condition from Theorem~\ref{thm:delay_passive} and is essentially an extension of \cite[Theorem 3.7]{Fridman2014} to the case of multiple delays.
\begin{proposition}
\label{prop:stable}
The pH system with delay terms \eqref{eq:pH_delay_simple} is globally exponentially stable for the delays $\tau_1\leq\ldots \leq\tau_q$ if there exists positive definite $\mathcal{P}$ and semi-definite matrices $\Theta_i$, $\Gamma_i$ and $\Omega_i$ such that 
\begin{equation}\label{eq:pH_NLMI_stable}
2{\rm sym}(\hat G_2^\top \mathcal{P} \hat G_1)+ \hat M_1+ \hat M_{2,\tau}+\hat M_{3,\tau}+\hat M_{4}<0,
\end{equation}
holds, where 
\begin{align*}
\hat G_1&=\begin{bmatrix}
    I_n & 0 & 0\\
    0 &0& \text{diag} (\tau_i I_n)_{i=1}^q
\end{bmatrix},  \hat G_2&=\begin{bmatrix}
\hat{A}&  \begin{smallbmatrix}
\hat Z_1 &\ldots & \hat Z_q\end{smallbmatrix} & 0  \\
 \begin{smallbmatrix} I_n & \ldots &I_n \end{smallbmatrix}^\top & -I_{n\cdot q}& 0 
\end{bmatrix},
\end{align*}
and the matrices $\hat M_1$, $\hat M_{2,\tau}$, $\hat M_{3,\tau}$, $\hat M_{4}$ are defined below 
\begin{align*}
 \hat M_1&= \begin{bmatrix}
      \sum_{i=1}^q \Theta_i& 0& 0  \\ 
     0 & -\mathrm{diag}(\Theta_i)_{i=1}^q & 0  \\ 0&0&0
\end{bmatrix},\\
\hat M_{2,\tau}&=\begin{bmatrix}
  \hat A^\top \Gamma_{\Sigma,\tau} \hat A & M_2^{12} & 0  \\
 *&  \begin{bmatrix}
\hat Z_i^\top   \Gamma_{\Sigma,\tau} \hat Z_j  
 \end{bmatrix}_{i,j=1}^q& 0  \\
 0&  0&  -{\rm diag}(\tau_i^2 \Omega_{i})_{i=1}^q
\end{bmatrix},\\
\hat M_{3,\tau}&=\begin{bmatrix}
 -\Gamma_{\Sigma,1}+ \Omega_{\Sigma,\tau} & \begin{bmatrix}
\Gamma_1 & \ldots & \Gamma_q 
\end{bmatrix} & 0  \\
 *& -\mathrm{diag}(\Gamma_i)_{i=1}^q & 0  \\
 0 & 0& 0 \\
\end{bmatrix},\\
\hat M_{4}&=-3\begin{bmatrix}
\Gamma_{\Sigma,1} & \begin{bmatrix}
 \Gamma_1 & \ldots & \Gamma_q 
\end{bmatrix}& -2\begin{bmatrix}
\Gamma_1 & \ldots \Gamma_q  
\end{bmatrix}&  \\
* & \mathrm{diag}(\Gamma_i)_{i=1}^q   & -2\mathrm{diag}(\Gamma_i)_{i=1}^q  \\
*& * & 4\mathrm{diag}( \Gamma_i)_{i=1}^q
\end{bmatrix}.   
\end{align*}
\end{proposition}
\begin{proof}
Repeating the calculations from the proof of Theorem~\ref{thm:delay_passive} shows that $\mathcal{S}$ is a Lyapunov function of \eqref{eq:pH_delay_simple} with $u=0$ which gives the asymptotic stability \cite{Fridman2014}. Using Proposition~\ref{prop:stability_equivalence}, we conclude the exponential stability.
\end{proof}

\section{Applications}
In this section, we apply our results on stability and passivity to two practical examples which are modeled as linear pH systems with delay terms of the form \eqref{eq:pH_delay}. 

\subsection{District heating networks}
In the following, we present an application of our passivity and stability results to delay equations arising in the modeling of district heating networks, where linear transport equations are transformed into delay equations \cite{Jensen23}. The delay is used to model the temperature transport in long pipes. Depending on the network topology and the location of the long pipes this leads to different delay systems. A schematics of the considered district heating network is shown in Figure~\ref{fig:district_heating_networks}. 
The supply heat $T_{\rm in}$ is produced via a boiler and a mixing loop: 
\begin{equation*}
V_{\rm s}\dot{T}_{\rm in}(t) = q_{\rm in}(t)\left(T_{\rm out}(t) - T_{\rm in}(t)\right) + P_h(t) -\kappa T_{\rm in}(t),
\end{equation*}
where $P_h(t)$ is the power provided by the boiler via the hot
feedwater and $T_{\rm out}$ is the return temperature from the network, which is a
simple mixing of the outlet temperatures of each consumer:
\begin{equation}
T_{\rm out}(t) = \frac{1}{q_{\rm in}(t)} \sum_{i=1}^{q} q_i(t) T_i(t),\quad  q_{\rm in}(t)=\sum_{i=1}^q q_i(t),\quad t\geq 0
\label{eq:Tout}
\end{equation}
where $q$ is the total number of consumers and the second equation represents the preservation of mass, 
i.e.\ the feedwater of the boiler is the sum of all consumer mass-flows $q_i(t)>0$ which are assumed to be positive for all times. This leads to 
\begin{align*}
V_{\rm s}\dot{T}_{\rm in}(t) &=  q_{\rm in}(t)\left(\frac{1}{q_{\rm in}(t)} \sum_{i=1}^{q} q_i(t) T_i(t) - T_{\rm in}(t)\right) -\kappa T_{\rm in}(t) + P_h(t)\\
&= \left(\sum_{i=1}^{q} q_i(t)( T_i(t) - T_{\rm in}(t))\right)-\kappa T_{\rm in}(t) + P_h(t).
\label{eq:Tin}
\end{align*}

The temperature at each consumer behaves according to the first-order
differential equation: 
\begin{equation}
V_i\dot{T}_i(t) = q_i(t)\left(T_{\rm in}(t - \tau_i) - T_i(t)\right) - w_i(t)
\label{eq:Ti}
\end{equation}
where $V_i$ is the effective volume of the $i$th heat exchanger and
$w_i(t)$ is the load at the $i$th consumer.

\begin{figure}[t]
    \centering
    \resizebox{\linewidth}{!}{%
    \begin{tikzpicture}[
        x=1cm,
        y=1cm,
        >=Latex,
        font=\footnotesize,
        producer/.style={
            draw,
            rounded corners,
            align=center,
            minimum width=2.7cm,
            minimum height=1.05cm,
            inner sep=2pt
        },
        consumer/.style={
            draw,
            rounded corners,
            align=center,
            minimum width=2.35cm,
            minimum height=0.85cm,
            inner sep=2pt
        },
        mix/.style={
            draw,
            circle,
            align=center,
            minimum size=0.75cm,
            inner sep=1pt
        },
        supply/.style={
            -{Latex[length=2mm]},
            thick,
            decorate,
            decoration={snake, amplitude=0.1mm, segment length=1.5mm}
        },
        return/.style={
            -{Latex[length=2mm]},
            thick,
            dashed
        },
        input/.style={
            -{Latex[length=2mm]},
            thick
        },
        loss/.style={
            -{Latex[length=2mm]},
            thick
        }
    ]

    \node[producer] (prod) at (0,0.2)
        {Producer \\ $V_{\rm s},\,T_{\rm in}$};

    \node[mix] (mix) at (0,-2.35) {$\sum$};

    \node[consumer] (c1) at (4.45,0.1)
        {Consumer $1$\\$V_1,\,T_1$};


    \node (dots) at (4.45,-0.75) {$\vdots$};

    \node[consumer] (cq) at (4.45,-1.75)
        {Consumer $q$\\$V_q,\,T_q$};

    \draw[input]
        ($(prod.north)+(0,0.75)$)
        -- node[right] {$P_h$}
        (prod.north);

    \draw[loss]
        (prod.west)
        -- ++(-0.95,0)
        node[left] {$\kappa T_{\rm in}$};

    \draw[input]
        (mix.north)
        -- node[left] {$T_{\rm out},\,q_{\rm in}$}
        (prod.south);

    \draw[supply]
        (prod.east)
        to[out=25,in=180]
        node[above,sloped,pos=0.55] {$q_1,\;\tau_1$}
        (c1.west);


    \draw[supply]
        (prod.east)
        to[out=-25,in=180]
        node[below,sloped,pos=0.25] {$q_q,\;\tau_q$}
        (cq.west);

    \draw[input]
        (c1.east)
        -- ++(0.9,0)
        node[right] {$w_1$};


    \draw[input]
        (cq.east)
        -- ++(0.9,0)
        node[right] {$w_q$};

    \draw[return]
        (c1.south west)
        to[out=-155,in=25]
        node[below,sloped,pos=0.15] {$q_1T_1$}
        (mix.east);


    \draw[return]
        (cq.west)
        to[out=180,in=-15]
        node[below,sloped,pos=0.25] {$q_qT_q$}
        (mix.east);



    \end{tikzpicture}%
    }
    \caption{District heating network with delayed supply transport. The producer supplies the temperature $T_{\rm in}$ to consumer $i$ through a pipe with mass flow $q_i$ and transport delay $\tau_i$. The consumer outlet temperatures are mixed into the return temperature $T_{\rm out}$.}
    \label{fig:district_heating_networks}
\end{figure}
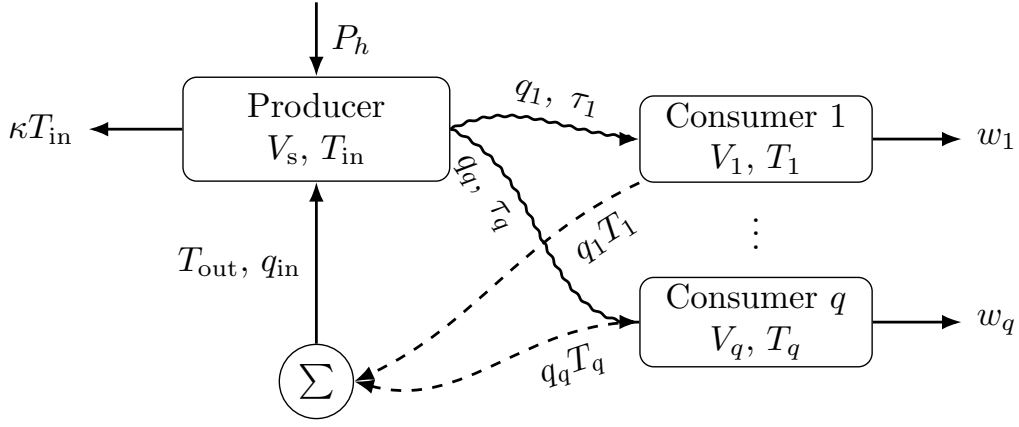

In the following, we assume that the hydraulics of the district heating system is stationary, i.e.\  the mass-flow rates $q_i>0$ for all $i=1,\ldots, q$ are constant and hence, by the mass conservation also $q_{\rm in}>0$ is constant. First, we investigate the case of constant mass-flows. Then the system can be written as a linear port-Hamiltonian delay system of the following form 
\begin{align*}
\tfrac{d}{dt}\begin{bmatrix}
    V_{\rm s}T_{\rm in}\\ V_1T_1\\ \vdots\\
    V_qT_q
\end{bmatrix}&=\underbrace{
\begin{bmatrix}
    -\sum_{i=1}^q q_i-\kappa &q_1& \ldots &q_q \\ 0 &-q_1& \ldots & 0 \\ \vdots & &\ddots & \vdots \\ 0&\ldots&& -q_q
\end{bmatrix}}_{=A} \underbrace{
\begin{bmatrix}
    V_{\rm s}^{-1} & 0 & \ldots & 0  \\ 0&V_1^{-1}& &\vdots\\  & & \ddots &  \\ 0&\ldots&&V_q^{-1}
\end{bmatrix}}_{=H}\begin{bmatrix}
    V_{\rm s}T_{\rm in}\\ V_1T_1\\ \vdots\\
    V_qT_q
\end{bmatrix}\\&~~~~+\begin{bmatrix}
    0\\ q_1T_{\rm in}(t-\tau_1) \\ \vdots \\ q_q T_{\rm in}(t-\tau_q) 
\end{bmatrix}+I\begin{bmatrix}
    P_h \\ -w_1 \\ \vdots \\ - w_q
\end{bmatrix}
\end{align*}

The simplified representation \eqref{eq:pH_delay_simple} is given by 
\begin{align}
\label{eq:simplified_heating}
\tfrac{d}{dt}\begin{bmatrix}
    \hat{T}_{\rm in}\\ \hat{T}_1\\ \vdots\\
    \hat{T}_q
\end{bmatrix}&=\underbrace{
\begin{bmatrix}
    \tfrac{-\sum_{i=1}^q q_i-\kappa}{V_{\rm s}} &\frac{q_1}{\sqrt{V_{\rm s}}\sqrt{V_1}}& \ldots &\frac{q_q}{\sqrt{V_{\rm s}}\sqrt{V_q}} \\ 0 &-\frac{q_1}{V_1}& \ldots & 0 \\ \vdots & &\ddots & \vdots \\ 0&\ldots&& -\frac{q_q}{V_q}
\end{bmatrix}}_{=\hat A} \!\!
\begin{bmatrix}
    \hat{T}_{\rm in}\\ \hat{T}_1\\ \vdots\\
    \hat{T}_q
\end{bmatrix}\\&~~~~+\begin{bmatrix}
    0\\ \frac{q_1}{\sqrt{V_{\rm s}}\sqrt{V_1}}\hat{T}_{\rm in}(t-\tau_1) \\ \vdots \\ \frac{q_q}{\sqrt{V_{\rm s}}\sqrt{V_q}} \hat{T}_{\rm in}(t-\tau_q) 
\end{bmatrix}+H^{1/2}\begin{bmatrix}
    P_h \\ -w_1 \\ \vdots \\ - w_q
\end{bmatrix}.
\end{align}

This can be written as a port-Hamiltonian system with delay terms of the form \eqref{eq:pH_delay_simple} by using 
\[
\hat J-\hat R=\hat A,\quad \hat Z_i:=\frac{q_i}{\sqrt{V_{\rm s}}\sqrt{V_i}}e_{i+1}e_1^\top.
\]

\begin{proposition}
 \label{prop:heating}
The system \eqref{eq:simplified_heating} is exponentially stable independently of the delay if $V_{\rm s}\geq \max_{i=1,\ldots,q}V_i$.
\end{proposition}
\begin{proof}
We aim to show delay-independent exponential stability by verifying the frequency sweeping conditions. First, note that, since $\hat A$ is upper triangular with negative diagonal entries it is Hurwitz. Moreover,
\[
\hat J-\hat R+\sum_{i=1}^q \hat Z_i=H^{1/2}\begin{smallbmatrix}
    -\sum_{i=1}^q q_i-\kappa &q_1& \ldots &q_q \\ q_1 &-q_1& \ldots & \vdots \\ \vdots & &\ddots & \vdots \\ q_q&\ldots&& -q_q
\end{smallbmatrix}H^{1/2}
\]
is Hurwitz which follows from the negative definiteness property for all $x\in\mathbb{R}^{q+1}$ with $x\neq 0$
\[
x^\top 
\begin{smallbmatrix}
    -\sum_{i=1}^q q_i-\kappa &q_1& \ldots &q_q \\ q_1 &-q_1& \ldots & \vdots \\ \vdots & &\ddots & \vdots \\ q_q&\ldots&& -q_q
\end{smallbmatrix}x=-\kappa x_1^2-\sum_{i=2}^{q+1}q_{i-1}(x_1-x_i)^2,
\]
which is negative.  The Sherman-Morrison-Woodbury formula with $u=e_1$ and $v=-[0, \frac{q_1}{\sqrt{V_{\rm s}}\sqrt{V_1}},\ldots, \frac{q_q}{\sqrt{V_{\rm s}}\sqrt{V_q}}]^\top  $ gives
\begin{align*}
&~~~~(j\omega -(\hat J-\hat R))^{-1}\begin{bmatrix}
   \hat Z_1&\ldots & \hat Z_q
\end{bmatrix}\\&=\begin{bmatrix}
j\omega+V_{\rm s}^{-1}(\sum_{i=1}^q q_i+\kappa) &-\frac{q_1}{\sqrt{V_{\rm s}}\sqrt{V_1}}& \ldots &-\frac{q_q}{\sqrt{V_{\rm s}}\sqrt{V_q}} \\ 0 & j\omega+ V_1^{-1}q_1& \ldots & 0 \\ \vdots & &\ddots & \vdots \\ 0&\ldots&& j\omega +V_q^{-1}q_q
\end{bmatrix}^{-1}\\&~~~~\cdot\begin{bmatrix}    
\frac{q_1}{\sqrt{V_{\rm s}}\sqrt{V_1}}e_{2}e_1^\top &\ldots & \frac{q_q}{\sqrt{V_{\rm s}}\sqrt{V_q}}e_{q+1}e_1^\top\end{bmatrix} \\
&=(D+uv^\top)^{-1}\begin{bmatrix}    
\frac{q_1}{\sqrt{V_{\rm s}}\sqrt{V_1}}e_{2}e_1^\top &\ldots & \frac{q_q}{\sqrt{V_{\rm s}}\sqrt{V_q}}e_{q+1}e_1^\top\end{bmatrix} \\
&=\left(D^{-1}-\frac{D^{-1}uv^\top D^{-1}}{1+v^\top D^{-1}u}\right)\begin{bmatrix}    
\frac{q_1}{\sqrt{V_{\rm s}}\sqrt{V_1}}e_{2}e_1^\top &\ldots & \frac{q_q}{\sqrt{V_{\rm s}}\sqrt{V_q}}e_{q+1}e_1^\top\end{bmatrix}\\
&=\begin{bmatrix} \frac{q_1}{\sqrt{V_{\rm s}}\sqrt{V_1}}\left(   
D^{-1}e_{2}e_1^\top  + \frac{\frac{q_1}{\sqrt{V_{\rm s}}\sqrt{V_1}}}{j\omega +V_1^{-1}q_1}\frac{1}{j\omega + V_{\rm s}^{-1}(\sum_{i=1}^q q_i+\kappa)}e_1e_1^\top\right) &\ldots \end{bmatrix},
\end{align*}
where we used in the last equation that $D$ is diagonal and therefore $v^\top D^{-1}u=0$. 

We compute the row-sum norm we have in the first row one non-zero entry in each of the matrices $\hat Z_i$ which are given by 
\begin{align*}
&~~~~\frac{\frac{q_1^2}{V_{\rm s}V_1}}{j\omega +V_1^{-1}q_1}\frac{1}{j\omega + V_{\rm s}^{-1}(\sum_{i=1}^q q_i+\kappa)}&=\frac{q_1}{q_1^{-1}V_1j\omega +1}\frac{1}{V_{\rm s}j\omega+\sum_{i=1}^q q_i+\kappa}.
\end{align*}
Computing the row sum of the absolute values leads to
\begin{align*}    
&\sum_{i=1}^q\left|\frac{q_i}{q_i^{-1}V_ij\omega +1}\frac{1}{V_{\rm s}j\omega+\sum_{i=1}^q q_i+\kappa}\right|&< \frac{1}{\sum_{i=1}^q q_i+\kappa}\left(\sum_{i=1}^q q_i\right)<1.
\end{align*}
In the remaining rows with index from $2$ to $n$ we have only a single non-zero entry with absolute value 
\[
\left|\frac{q_i}{\sqrt{V_{\rm s}}\sqrt{V_i}}
\frac{1}{j\omega +V_i^{-1}q_i}\right|< \frac{\sqrt{V_i}}{\sqrt{V_{\rm s}}}\leq 1, \quad i=1,\ldots,n,
\]
where we used that $\omega\neq 0$. 
Hence, the exponential stability of the system follows from Proposition~\ref{prop:sweeping}.
\end{proof}

\begin{remark}
The previous network topology assumes that the temperature delay occurs only during transport to the consumers but not in the output temperature in the return flow. However, in many practical district heating system the water also has to be transported back to the producer which induces additional delay terms. The proof of Proposition~\ref{prop:heating} this would lead to higher rank perturbations and therefore the result cannot be generalized directly to these situations. However, the passivity and stability properties can still be verified numerically by solving the LMIs that are presented in Theorem~\ref{thm:delay_passive} and Proposition~\ref{prop:stable}. 
\end{remark}

\subsection{Harmonic oscillator with delayed stiffness}
We consider a harmonic oscillator with state $x=(q,p)$ where $q$ is the generalized position and $p=m\dot q$ is the momentum and an input force $u$. Furthermore, we assume a velocity-dependent damping $d>0$ and a delayed stiffness described by some $a>0$
\begin{align}
\dot x(t)&=\begin{bmatrix}
0&1 \\ -1&-d
\end{bmatrix}Hx(t)+\begin{bmatrix}
0&0\\ -a&0
\end{bmatrix}Hx(t-\tau)+\begin{bmatrix}
    0 \\ 1
\end{bmatrix}u(t) \nonumber \\
y(t)&= \begin{bmatrix}
    0 & 1
\end{bmatrix}\begin{bmatrix}
    k&0\\0&m^{-1}
\end{bmatrix} x(t),\quad H=\begin{bmatrix}
    k&0\\0&m^{-1}
\end{bmatrix}. \label{eq:harmonic_stiff}
\end{align}
Hence, the rescaled pH system \eqref{eq:pH_delay_simple} is given by 
\begin{align}\label{eq:harmonic_stiff_simple}
\dot{\hat{x}}(t)&=\begin{bmatrix}
        0&\alpha\\
        -\alpha&-\beta
    \end{bmatrix}\hat{x}(t) +\begin{bmatrix}
        0&0\\
        -a\alpha&0
    \end{bmatrix}\hat{x}(t-\tau)+\begin{bmatrix}
    0 \\ b
\end{bmatrix}u(t) \nonumber \\
y(t)&= \begin{bmatrix}\nonumber
    0 & b
\end{bmatrix}\hat{x}(t)
\end{align}
where we use 
\[
    \alpha:=\sqrt{\frac{k}{m}},
    \qquad
    \beta:=\frac{d}{m},
    \qquad
    b:=m^{-1/2}.
\]
From Theorem~\ref{thm:delay_passive} we find that the system \eqref{eq:harmonic_stiff} is passive with the storage function \eqref{eq:storage_fixed_delay_PH} with  
\begin{align*}
    \mathcal{P}&=\begin{bmatrix}
        P&0\\0&0
    \end{bmatrix},\quad P=\frac{1}{2}\begin{bmatrix}
        1+a &0\\0&1
    \end{bmatrix},\quad \Theta_1=0, \Gamma_1&=\frac{a}{2\tau}\begin{bmatrix}
        1&0\\0&0
    \end{bmatrix},\quad \Omega_1=0 
\end{align*}
whenever $0\leq ka\tau\leq d$ which we will verify in the following. 
The matrices $G_1$ and $G_2$ are
\[
    G_1=
    \begin{bmatrix}
        1&0&0&0&0&0&0\\
        0&1&0&0&0&0&0\\
        0&0&0&0&\tau&0&0\\
        0&0&0&0&0&\tau&0
    \end{bmatrix},
\]
and
\[
    G_2=
    \begin{bmatrix}
        0&\alpha&0&0&0&0&0\\
        -\alpha&-\beta&-a\alpha&0&0&0&b\\
        1&0&-1&0&0&0&0\\
        0&1&0&-1&0&0&0
    \end{bmatrix}.
\]
With the above choice, we find 
\[
    2\operatorname{sym}(G_2^\top\mathcal P G_1)
    =
    \begin{bmatrix}
        0&\frac{a\alpha}{2}&0&0&0&0&0\\
        \frac{a\alpha}{2}&-\beta&-\frac{a\alpha}{2}&0&0&0&\frac b2\\
        0&-\frac{a\alpha}{2}&0&0&0&0&0\\
        0&0&0&0&0&0&0\\
        0&0&0&0&0&0&0\\
        0&0&0&0&0&0&0\\
        0&\frac b2&0&0&0&0&0
    \end{bmatrix}.
\]

Since $\Theta=0$, we have 
    $M_1=0$. Furthermore, 
    $M_{2,\tau}
    =\frac{a\tau\alpha^2}{2}  e_2e_2^\top$ where $e_2$ is the canonical unit vector in $\mathbb{R}^7$.
The corrected endpoint term is
\[
    M_{3,\tau}
    = \frac{a}{2\tau}
    \begin{bmatrix}
        -1&0&1&0&0&0&0\\
        0&0&0&0&0&0&0\\
        1&0&-1&0&0&0&0\\
        0&0&0&0&0&0&0\\
        0&0&0&0&0&0&0\\
        0&0&0&0&0&0&0\\
        0&0&0&0&0&0&0
    \end{bmatrix}.
\]

The Wirtinger correction term is
\[
    M_4 
    = \frac{a}{2\tau}
    \begin{bmatrix}
        -3&0&-3&0&6&0&0\\
        0&0&0&0&0&0&0\\
        -3&0&-3&0&6&0&0\\
        0&0&0&0&0&0&0\\
        6&0&6&0&-12&0&0\\
        0&0&0&0&0&0&0\\
        0&0&0&0&0&0&0
    \end{bmatrix}.
\]
the supply-rate term is
$    M_5
    =
    -\frac b2(e_2e_7^\top+e_7 e_2^\top )$.  
Therefore,
\begin{align*}
   & ~~~~2\operatorname{sym}(G_2^\top\mathcal P G_1)
    +M_1+M_{2,\tau}+M_{3,\tau}+M_4+M_5\\ 
 &=   \begin{bmatrix}
        -\frac{2a}{\tau}
        &
        \frac{a\alpha}{2}
        &
        -\frac{a}{\tau}
        &
        0
        &
        \frac{3a}{\tau}
        &
        0
        &
        0
        \\
        \frac{a\alpha}{2}
        &
        \frac{a\tau\alpha^2}{2}-\beta
        &
        -\frac{a\alpha}{2}
        &
        0
        &
        0
        &
        0
        &
        0
        \\
        -\frac{a}{\tau}
        &
        -\frac{a\alpha}{2}
        &
        -\frac{2a}{\tau}
        &
        0
        &
        \frac{3a}{\tau}
        &
        0
        &
        0
        \\
        0&0&0&0&0&0&0
        \\
        \frac{3a}{\tau}
        &
        0
        &
        \frac{3a}{\tau}
        &
        0
        &
        -\frac{6a}{\tau}
        &
        0
        &
        0
        \\
        0&0&0&0&0&0&0
        \\
        0&0&0&0&0&0&0
    \end{bmatrix}.
\end{align*}
To verify the the negative semi-definiteness, we remove the zero rows and columns and consider the reduced matrix  
\[ 
\begin{bmatrix} -\frac{2a}{\tau} & \!\! \frac{a\alpha}{2} &\!\! -\frac{a}{\tau} &\!\! \frac{3a}{\tau} \\ \frac{a\alpha}{2} & \!\!\frac{a\tau\alpha^2}{2}-\beta & \!\!-\frac{a\alpha}{2} &\!\! 0 \\ -\frac{a}{\tau} & \!\!-\frac{a\alpha}{2} & \!\!-\frac{2a}{\tau} & \!\!\frac{3a}{\tau} \\ \frac{3a}{\tau} & \!\! 0 & \!\! \frac{3a}{\tau} & \!\! -\frac{6a}{\tau} \end{bmatrix}
\!=\!T^\top\!\begin{bmatrix} -\frac{a}{2\tau}&0&0&0\\ 0&-\frac{3a}{2\tau}&0&0\\ 0&0&a\tau\alpha^2-\beta&0\\ 0&0&0&0 \end{bmatrix} \! T 
\] 
which can be diagonalized using 
\[ T= \begin{bmatrix} 1&-\tau\alpha&-1&0\\ 1&0&1&-2\\ 0&1&0&0\\ 0&0&1&0 \end{bmatrix}. \] 
Hence, the passivity holds provided that \[ a\tau\frac{k}{m}-\frac{d}{m}=a\tau\alpha^2-\beta\leq 0.\]
Hence, the system \eqref{eq:harmonic_stiff} is passive if $\tau\in\left[0,\frac{d}{ak}\right]$. 

Let $d=k=a=1.$ The numerical resolution of the 
LMI using YALMIP \cite{Lofberg2004} with the MOSEK solver \cite{mosek} yields $\tau \in \left[0,1.0002\right],$ for $\Omega_1=0,$ and  $\tau \in 
\left[0,0.009\right],$ for $\Gamma_1=0.$ This shows that the presence of the $\Gamma_1$-term leads to less conservative results.

\subsection{Harmonic oscillator with delayed damping}
A harmonic oscillator with delayed damping $d>0$ can be written as a pH system with delay terms of the form \eqref{eq:pH_delay}
\begin{align}
\nonumber
\dot x(t)&=\begin{bmatrix}
0&1 \\ -1&0
\end{bmatrix}Hx(t)+\begin{bmatrix}
0&0\\ 0&-d
\end{bmatrix}Hx(t-\tau)+\begin{bmatrix}
    0 \\ 1
\end{bmatrix}u(t)\\
y(t)&= \begin{bmatrix}
    0 & 1
\end{bmatrix} H x(t),\quad H:=\begin{bmatrix}
    k &0 \\ 0& m^{-1}
\end{bmatrix}. \label{eq:harmonic}
\end{align}
This system has $\hat{R}=0$ and therefore the frequency sweeping condition from Proposition~\ref{prop:sweeping}~(i) is violated. Thus, the system is not stable for all delays and it is also not port-Hamiltonian in the sense of \cite{BreitHU24}. In the following, we use the LMI conditions from Proposition~\ref{prop:stable} to determine the range of $\tau$ where the system is stable and compare them with the exact value of $\tau$ which can be determined by solving the characteristic equation \eqref{eq:characteristic} by searching for solutions on the imaginary axis. 
\vspace{2mm}\\
The characteristic equation of the zero-input system \eqref{eq:harmonic}, i.e.\ $u=0$, is given by 
\[
\lambda^2+dm^{-1}\lambda e^{-\lambda \tau} +km^{-1}=0
\]
We look for  purely imaginary roots $\lambda= i \omega, \hspace{2mm} \omega \in \mathbb{R}$ with $\omega\neq 0$. 
Substitute into the characteristic equation, we obtain 
\begin{align}
    \label{eq:characteristic_example}
    k - m\omega^2 + d\omega\sin(\omega\tau)=0 \quad \text{and} \quad d\omega\cos(\omega\tau)=0
\end{align}
which implies that 
$ \cos(\omega \tau)=0$  hence $\omega \tau = \tfrac{\pi}{2} + n\pi, \hspace{2mm} n \in \mathbb{Z}$. Then $\sin(\omega \tau)= (-1)^n$ and considering 
the real-part of \eqref{eq:characteristic_example} leads to 
\[
k-m\omega^2+d\omega (-1)^n=0.
\]
Since $d,m>0,$ the first crossing occurs exactly at  
\[
\tau_0=\tfrac{\pi/2}{\tfrac{d + \sqrt{d^2 + 4mk}}{2m}}.
\]
We conclude that, system \eqref{eq:harmonic} is exponentially stable for $\tau \in \left[0, \tfrac{\pi m }{d + \sqrt{d^2 + 4mk}}\right).$\\
In order to check the feasibility of the proposed LMI numerically, we set $k=m=1.$ In this case, system \eqref{eq:harmonic} is exponentially stable for $\tau \in [0, 0.9708).$

Regarding the passivity of \eqref{eq:harmonic} we observe that for $\tau=0$ the system is a classical port-Hamiltonian system and therefore passive~\cite{CherifiGernandtHinsen2024}.  Looking at the poles of the transfer function we can find that the transfer function is not positive real for all $\tau>0$ and therefore using Proposition~\ref{prop:positive real} 
 the system \eqref{eq:harmonic} is not passive.

\begin{table}[htbp]
    \centering
    \caption{Stability interval $\tau \in \left[ \tau_{min} , \tau_{max}\right] $}
    \label{tab:passivity}
    \begin{tabular}{ccccc}
        \toprule
    Storage function &   $\tau_{min}$ & $\tau_{max}$  \\
        \midrule
       $\mathcal{S}$ given by \eqref{eq:storage_fixed_delay_PH} &0  & 0.9687\\  
      $\mathcal{S},$ with $\mathcal{P}=\begin{bmatrix}
           Q_{11}& 0\\
           0&Q_{22}
       \end{bmatrix}$ &0 & 0.8660\\ 
        $\mathcal{S}$ with $\Gamma_1=0$& 0  & 0.8660 \\
        $\mathcal{S}$ with $\Omega_1=0$&0  &0.9687 \\
      $\mathcal{S}$ with $\Omega_1=\Gamma_1=0$&0  & 0.0001\\

        \bottomrule
    \end{tabular}
\end{table}

\section{Conclusion}
We have formulated delay-dependent and independent passivity conditions for linear 
port-Hamiltonian systems with multiple constant delays in terms of linear matrix 
inequalities. We further proved that this passivity property is preserved under negative 
feedback interconnections, retaining one of the key structural properties of 
port-Hamiltonian modeling. The proposed framework is further supported by an application to district heating networks and to a harmonic oscillator with delayed damping. Future work may focus on extending these results to 
nonlinear delay systems, including distributed and time-varying delays, as well as 
discrete-time delay systems.

\section*{Acknowledgments}
This work was funded by the Deutsche Forschungsgemeinschaft (DFG, German
Research Foundation) – Project-ID 531152215 – CRC~1701.

\bibliographystyle{abbrv}
\bibliography{references}

\end{document}